\documentclass[11pt,twoside]{amsart}

\usepackage{color}
\usepackage{verbatim}
\usepackage{amsmath}
\usepackage{amsfonts,amsthm,amssymb}
  \usepackage{mathtools}

\newtheorem{thm}{Theorem}[section]

\newtheorem{remark}{Remark}[section]

\def\d{\partial}

\newcommand{\with}{\quad\hbox{with}\quad}
\newcommand{\andf}{\quad\hbox{and}\quad}

\def\wt{\widetilde}

\renewcommand\H{\mathbb{H}}
\newcommand\R{\mathbb{R}}

\renewcommand\L{\mathbb{L}}

\newcommand{\N}{\mathbb{N}}

\renewcommand{\div}{\mbox{\rm div}\;\!}

\def\cH{{\mathcal H}}

\def\cL{{\mathcal L}}

\def\cR{{\mathcal R}}

\def\d{\partial}

\begin{document}
\title[Decay for NS2D]{On the decay and Gevrey regularity of the solutions to the Navier-Stokes equations in general two-dimensional domains}


\author[R. Danchin]{Rapha\"el Danchin}
\address[R. Danchin]{Univ Paris Est Creteil, Univ Gustave Eiffel, CNRS, LAMA UMR8050, F-94010 Creteil, France
and Sorbonne Universit\'e, LJLL UMR 7598, 4 Place Jussieu, 75005 Paris}
\email{danchin@u-pec.fr}

\maketitle


\begin{small}\begin{center}\textbf{Abstract}\end{center}
 The present paper is devoted to the proof of time decay estimates for
derivatives at any order of finite energy global solutions of the Navier-Stokes equations in general two-dimensional domains.
These estimates only depend on the order of derivation and on the $L^2$ norm of the initial data.
The same elementary method just based on energy estimates and Ladyzhenskaya inequality also leads
to Gevrey regularity results.
\bigbreak

\begin{center}\textbf{R\'esum\'e}\end{center}
On s'int\'eresse aux  propri\'et\'es  de décroissance temporelle pour les dérivées des solutions globales à énergie finie des équations de Navier-Stokes dans des domaines généraux bidimensionnels.
Les estimations obtenues  ne dépendent que de l'ordre de dérivation et de la norme $L^2$ des données initiales.
La même méthode élémentaire basée sur les bornes d'énergie et l'inégalité de Ladyzhenskaya conduit également à des résultats de régularité Gevrey.
\bigbreak\noindent\textbf{Keywords:}  Incompressible Navier-Stokes equations, two-dimensional, decay estimates, Gevrey regularity.
\end{small}
\vspace{1cm}

We are concerned with the  
incompressible Navier-Stokes equations that govern the evolution of the velocity field $u=u(t,x)$ and 
pressure function $P=P(t,x)$ of homogeneous incompressible viscous flows
in a general domain $\Omega$ of $\R^2$ or in a two-dimensional periodic box. 
 Adopting standard notation these equations read
 \begin{equation*}
\left\{\begin{aligned}
 &u_{t}+\div(u\otimes u)-\Delta u+\nabla P=0&&\quad\hbox{in }\ \R_+\times\Omega,  \\
&\div u=0&&\quad\hbox{in }\ \R_+\times\Omega,\\
&u|_{t=0}=u_0&&\quad\hbox{in }\Omega.
\end{aligned}\right.\leqno(NS)\end{equation*}
The initial data $u_0$ is a given divergence free vector-field with 
 normal component vanishing at the boundary $\d\Omega$ of $\Omega$ and 
we supplement (NS) with homogeneous Dirichlet boundary conditions for $u$ at $\d\Omega.$ 
\smallbreak 
The global existence theory  for (NS) originates from  the 
paper \cite{Leray2} by J. Leray in 1934.  In the case $\Omega=\R^3,$ by combining   the   energy balance associated to (NS):
\begin{equation}\label{eq:L2NS}
\frac12\|u(t)\|_{L^2}^2+\int_0^t\|\nabla u\|_{L^2}^2\,d\tau =  \frac12 \|u_0\|_{L^2}^2,\qquad t\in\R_+,
\end{equation}
with compactness arguments, he succeeded in  constructing  for any  divergence free 
$u_0$ in $L^2(\R^3;\R^3)$    a global  distributional solution of (NS) 
satisfying \eqref{eq:L2NS} \emph{with an inequality} (viz. the left-hand side is bounded by the right-hand side). 

Leray's result turns out to be very robust and can be adapted to any two or three-dimensional domain: we have
the following statement that is proved in e.g. \cite{CDGG}:
\begin{thm}\label{thm:0} Let $\Omega$ be  a domain of $\R^d$ (with $d=2,3$) 
and denote by $L^2_\sigma(\Omega)$  the completion of the set ${\mathcal V}_\sigma$ of smooth divergence free vector-fields
compactly supported in $\Omega$ for the $L^2(\R^d;\R^d)$ norm.  Let 
$H^1_{0,\sigma}(\Omega)$ be the completion of ${\mathcal V}_\sigma$  for the $H^1(\R^d;\R^d)$ norm.

Then, for any $u_0\in L^2_\sigma(\Omega)$ there exists a global distributional solution $(u,P)$ of (NS) 
with $u\in L^\infty(\R_+;L^2_\sigma(\Omega))\cap L^2_{loc}(\R_+;H^1_{0,\sigma}(\Omega))$
satisfying 
\begin{equation}\label{ineq:L2NS}
\frac12\|u(t)\|_{L^2}^2+\int_0^t\|\nabla u\|_{L^2}^2\,d\tau \leq  \frac12 \|u_0\|_{L^2}^2,\qquad t\in\R_+.
\end{equation}
\end{thm}
So far,  uniqueness of Leray's solutions in dimension three is an open question. 
In contrast, it holds true in dimension two (see the works by O.A. Ladyzhenskaya in \cite{Lady}, and by J.-L. Lions and G. Prodi  in \cite{LP}). 
  The key to the proof was the following \emph{Ladyzhenskaya  inequality}  
 \begin{equation}\label{eq:lad}\|z\|_{L^4}^2\leq C_0\|z\|_{L^2}\|\nabla z\|_{L^2},\qquad z\in H^1_0(\Omega)\end{equation}
 that will also play  a decisive role in the present paper. 
 \medbreak
 Since the pioneering work by J. Leray, a huge amount of literature has been devoted to the study of (NS) both in two and
 three dimensional domains.  Our goal here is to derive $L^2$ decay estimates for time derivatives at any order of two-dimensional finite energy global solutions. We shall see that our method actually gives for free  Gevrey regularity 
 for short time (or all time if the data are small). 
 \medbreak 
 Exhibiting time decay estimates for smooth and small solutions of (NS) goes back 
 to the papers by S. Kawashima, A. Matsumura and T. Nishida \cite{KMT} and J.G. Heywood \cite{H}
 devoted to the case $\Omega=\R^3$.  
 In both papers, in addition to be smooth enough, the initial velocity is required to be globally integrable on $\R^3.$
 An important breakthrough has been made by M.  Schonbek \cite{Sch1} in 1985  who observed that
 \emph{any} weak solution supplemented with an initial velocity  $u_0$ in $L^1(\R^3)\cap L^2(\R^3)$ 
satisfies time decay estimates. More accurate decay rates have been obtained shortly after 
by R. Kajikiya and T. Miyakawa \cite{KM} and M. Wiegner \cite{W}.
In \cite{Sch2},  M.  Schonbek pointed out that one  cannot expect  any generic  
 rate of decay for $\|u(t)\|_{L^2}$ if the initial data is only in $L^2.$
 \smallbreak
 It is also worth mentioning works pointing out the Gevrey or even analyticity of the solutions to (NS). 
For exemple, for periodic boundary conditions, C. Foias and R. Temam proved in \cite{FT} that $H^1$ data
give rise to  solutions with analytic regularity in time,  globally in time in dimension two, and locally in time 
in dimension three. This result has been adapted to the whole space setting and considerably refined
by P.-G. Lemari\'e-Rieusset \cite{Lem} then   by M. Oliver and E. Titi in \cite{OT}, and translated in the language
of critical Besov spaces ($u_0\in \dot B^{-1+3/p}_{p,q}(\R^3)$ with $1\leq p<\infty$ and $1\leq q\leq\infty$) 
by H. Bae, A. Biswas and E. Tadmor in \cite{BBT}.  By a different approach, J.-Y. Chemin in \cite{Chemin}
obtained (space) analyticity estimates of $L^2$ type in the case of small data (see also \cite{CGZ}). 
The more complicated case of the Navier-Stokes equations with potential forces has  been investigated
by several authors. The reader may in particular refer to the survey paper by C. Foias, L. Huan and J.-C. Saut \cite{FHS}
where asymptotic expansions for large time are presented.  
\smallbreak
Most of the aforementioned worked dedicated to decay estimates strongly rely on Fourier or spectral analysis.
In particular,  the \emph{Fourier splitting method}  of M. Schonbek  \cite{Sch1} can hardly be adapted
to  general domains (or at the price of complicated arguments that require the domain to be smooth, see
 \cite{BM}). 
Here we shall see that using  only the energy method and    
  Ladyzhenskaya inequality \eqref{eq:lad} leads to optimal time decay estimates. 
  
  In order to give an idea of our approach,  let us 
  consider the linearized version of (NS) about a null solution, namely the following 
  evolutionary Stokes equations: 
   \begin{equation}\label{eq:stokes}
\left\{\begin{aligned}
 &u_{t}-\Delta u+\nabla P=0&&\quad\hbox{in }\ \R_+\times\Omega,  \\
&\div u=0&&\quad\hbox{in }\ \R_+\times\Omega,\\
&u|_{t=0}=u_0&&\quad\hbox{in }\Omega.
\end{aligned}\right.\end{equation}
 Let us explain  how  to bound just in terms of $\|u_0\|_{L^2}$ 
 and  by elementary arguments (that are valid in any domain) the following quantities  for all $k\in\N$:
\begin{equation}\label{def:LH}\begin{aligned}
\cL_{2k}(t):= \|t^k u_t^{(k)}(t)\|_{L^2}&\andf \cH_{2k}(t):=\|t^k \nabla u_t^{(k)}(t)\|_{L^2},\\
\cL_{2k+1}(t):= \|t^{k+\frac12} \nabla u_t^{(k)}(t)\|_{L^2}&\andf \cH_{2k+1}(t):=\|t^{k+\frac12} u_t^{(k+1)}(t)\|_{L^2},
\end{aligned}\end{equation}
where $u_t^{(k)}$ stands for the $k$-th time derivative of $u.$ 
\medbreak
To handle the case of even exponents, we start from 
$$\partial_t(t^k u_t^{(k)}) - \Delta(t^ku_t^{(k)}) +\nabla(t^k P_t^{(k)})= kt^{k-1} u_t^{k}$$
then take the $L^2$ scalar product with $t^ku_t^{(k)}$ and perform an integration by parts to get
$$\frac12\frac d{dt} \cL_{2k}^2+\cH_{2k}^2=k\cH_{2k-1}^2. $$
For odd exponents, we rather take   the $L^2$ scalar product with $t^{k+1}u_t^{(k+1)}$ and get
$$\frac12\frac d{dt} \cL_{2k+1}^2+\cH_{2k+1}^2=\biggl(k+\frac12\biggr)\cH_{2k}^2. $$
In short, we have for all $m\in\N,$
$$\frac12\frac{d}{dt} \cL_m^2+\cH_m^2=\frac m2\cH_{m-1}^2$$
which immediately leads after  summation on $m$ and time integration to 
\begin{equation}\label{est:stokes}
\sum_{m=0}^\infty\frac{\cL_m^2(t)}{m!} + \int_0^t \sum_{m=0}^\infty\frac{\cH_m^2(\tau)}{m!}\,d\tau
=\|u_0\|_{L^2}^2,\qquad t\in\R_+.\end{equation}
We shall proceed in the same way for the Navier-Stokes system, treating the nonlinear term 
by combination of H\"older and   Ladyzhenskaya inequalities  (this is the only place where
dimension two comes into play). 
This   will lead to the following results:
  \begin{itemize}
  \item[---]  Gevrey type regularity (almost as good as \eqref{est:stokes}), that implies time decay estimates for
  derivatives of arbitrary order    in the case of small initial data;
  \item[---]   decay estimates at any order, in terms of $\|u_0\|_{L^2}$ for general finite energy solutions;
  \item[---]  small time Gevrey regularity in the case of large data;
  \item[---]  faster decay for all derivatives in case it is known beforehand that $\|u(t)\|_{L^2}$ has some algebraic decay. 
  \end{itemize}
  We conclude this introduction pointing out that  we here only considered the decay of time derivatives 
  both for simplicity  and because proving similar results for the space derivatives  
  requires the fluid domain to have enough  smoothness. The reader may refer to  Remark \ref{r:spacereg} for
  a short development on this issue.


  \section{The case of small data}

  The main goal of this section is to prove the following theorem.
  \begin{thm} Let $\alpha>0.$ There exists a constant $c_\alpha$ depending only on $\alpha$ such that 
  for any data $u_0$ in $L^2_\sigma$ satisfying $\|u_0\|_{L^2}\leq c_\alpha,$
    the corresponding global finite energy solution $u$ satisfies
  \begin{multline}\label{est:smalldata}\sum_{k=0}^{\infty} \biggl(\frac{t^{2k}}{2^{2k}(k!)^{2+\alpha}} \|u_t^{(k)}(t)\|_{L^2}^2+
\frac{t^{2k+1}}{2^{2k+1} k! ((k+1)!)^{1+\alpha}} \|\nabla u_t^{(k)}(t)\|_{L^2}^2\biggr)\\+
\frac12\sum_{k=0}^\infty \int_0^t\biggl(\frac{\tau^{2k}}{2^{2k}(k!)^{2+\alpha}}\|\nabla u_\tau^{(k)}(\tau)\|_{L^2}^2
+\frac{\tau^{2k+1}}{2^{2k+1} k! ((k+1)!)^{1+\alpha}} \|u_\tau^{(k+1)}(\tau)\|_{L^2}^2\biggr)d\tau
\leq  \|u_0\|_{L^2}^2.\end{multline}
  \end{thm}
  \begin{proof}
Here and in the following sections, we concentrate on the proof of a priori estimates. 
The underlying idea is that one can get exactly the same bounds 
for any  approximation that relies on the use of spectral orthogonal projectors 
(like e.g. the Galerkin method) and that following the compactness procedure that 
is used for proving Theorem \ref{thm:0} ensures that the solution that is constructed in this way 
satisfies the announced inequalities. 

\smallbreak
Now, with the notation introduced in \eqref{def:LH}, the energy balance  \eqref{eq:L2NS} translates into 
\begin{equation}\label{eq:H0}
\frac12\cL_0^2(t)+\int_0^t\cH_0^2\,d\tau=\frac12\|u_0\|_{L^2}^2.
\end{equation}
To handle $\cL_m$ and $\cH_m$ in the case of  odd index $m,$ we apply  $\d_t^{k-1}$ (for any $k\geq1$)
 to (NS) and use Leibniz formula, getting:
$$u_t^{(k)} -\Delta u_t^{(k-1)}+\nabla P_t^{(k-1)} =-\sum_{j=0}^{k-1} \binom{k-1}{j} 
\div\bigl(u_t^{(j)}\otimes u_t^{(k-1-j)}\bigr)\cdotp$$
Taking the scalar product with $t^{2k-1} u_t^{(k)}$ and integrating by parts where needed yields
$$\displaylines{\|t^{k-\frac12} u_t^{(k)}\|_{L^2}^2+\int_\Omega t^{2k-1}\nabla u_t^{(k-1)}\cdot\d_t\nabla u_t^{(k-1)}\,dx
=-\sum_{j=0}^{k-1} \binom{k-1}{j}\cR_{j,2k-1}\hfill\cr\hfill
\with \cR_{j,2k-1}:= \int_\Omega  \div\bigl(t^ju_t^{(j)}\otimes(t^{k-1-j}u_t^{(k-1-j)})\bigr)\cdot (t^{k} u_t^{(k)})\,dx,}$$
whence
\begin{equation}\label{eq:2k-1}
\frac12 \frac d{dt} \cL_{2k-1}^2 +\cH_{2k-1}^2=\biggl(k-\frac12\biggr)\cH_{2k-2}^2- \sum_{j=0}^{k-1} \binom{k-1}{j} \cR_{j,2k-1}.\end{equation}
For  all $j\in\{0\cdots,k-1\},$ performing an integration by parts gives
$$\cR_{j,2k-1}:=  - \int_\Omega  \bigl(t^ju_t^{(j)}\otimes(t^{k-1-j}u_t^{(k-1-j)})\bigr)\cdot
 (t^{k} \nabla u_t^{(k)})\,dx,$$  whence using H\"older and Ladyzhenskaya inequality and the definition of $\cL_m$ and $\cH_m,$
 $$\begin{aligned}
  \cR_{j,2k-1} &\leq \|t^j u_t^{(j)}\|_{L^4}\|t^{k-1-j}u_t^{(k-1-j)}\|_{L^4} \|t^k\nabla u_t^{(k)}\|_{L^2}\\
   &\leq C_0\cL_{2j}^{1/2}\cH_{2j}^{1/2}\cL_{2k-2-2j}^{1/2}\cH_{2k-2-2j}^{1/2}\cH_{2k}.
   \end{aligned}$$
 Hence we have 
\begin{equation}\label{eq:cL2k-1b}
\frac12 \frac d{dt} \cL_{2k-1}^2 +\cH_{2k-1}^2\leq
\biggl(k-\frac12\biggr)\cH_{2k-2}^2+C_0 \sum_{j=0}^{k-1} \binom{k\!-\!1}{j} \cL_{2j}^{1/2}\cH_{2j}^{1/2}\cL_{2k-2-2j}^{1/2}
\cH_{2k-2-2j}^{1/2}\cH_{2k}.\end{equation}
In order to handle  even indices, we apply  $t^k\d_t^k$ to (NS). Using Leibniz formula yields:
$$\displaylines{\d_t(t^k u_t^{(k)}) + \div\bigl(u\otimes t^k u_t^{(k)}\bigr)-\Delta(t^k u_t^{(k)})+\nabla(t^kP_t^{(k)}) 
\hfill\cr\hfill= kt^{k-1} u_t^{(k)}
-\sum_{j=1}^k \binom{k}{j} \div\bigl(t^ju_t^{(j)}\otimes (t^{k-j}u_t^{(k-j)})\bigr)\cdotp}$$
Hence taking the $L^2$ scalar product with $t^ku_t^{(k)}$  and performing suitable 
integration by parts gives:
\begin{multline}\label{eq:2k}
\frac12 \frac d{dt} \cL_{2k}^2 +\cH_{2k}^2 = k \cH_{2k-1}^2 - \sum_{j=1}^k \binom{k}{j} \cR_{j,2k}\\
\with \cR_{j,2k}:=\int_\Omega \div\bigl(t^ju_t^{(j)}\otimes(t^{k-j}u_t^{(k-j)})\bigr)\cdot (t^k u_t^{(k)})\,dx.\end{multline}
Observe that 
\begin{equation}\label{eq:Rj2k}
 \cR_{j,2k}:=- \int_\Omega  \bigl(t^ju_t^{(j)}\otimes(t^{k-j}u_t^{(k-j)})\bigr)\cdot (t^{k}\nabla u_t^{(k)})\,dx.\end{equation}
Therefore, combining H\"older inequality and \eqref{eq:lad} gives  
 $$\begin{aligned} \cR_{j,2k} &\leq
 \|t^ju_t^{(j)}\|_{L^4} \|t^{k-j}u_t^{(k-j)}\|_{L^4} \|t^k \nabla u_t^{(k)}\|_{L^2}\\
 &\leq C_0\cL_{2j}^{1/2}\cH_{2j}^{1/2}\cL_{2k-2j}^{1/2}\cH_{2k-2j}^{1/2}\cH_{2k}.\end{aligned}$$
   Hence we have 
   \begin{equation}\label{eq:cL2kb}
\frac12 \frac d{dt} \cL_{2k}^2 +\cH_{2k}^2 \leq k \cH_{2k-1}^2 + C_0\sum_{j=0}^{k-1} \binom{k}{j}
    \cL_{2j}^{1/2}\cH_{2j}^{1/2}\cL_{2k-2j}^{1/2}\cH_{2k-2j}^{1/2}\cH_{2k}.
    \end{equation}
      Let us use renormalize the functionals $\cL_m$ and $\cH_m$ as follows: 
       \begin{equation}\label{eq:1ren}
\wt  L_{2k}:=\frac{\cL_{2k}}{2^kk!},\!\!\quad \wt H_{2k}:=\frac{\cH_{2k}}{2^kk!},\!\!\quad
 \wt L_{2k-1}:=\frac{\sqrt2\,\cL_{2k-1}}{2^{k}\sqrt{(k\!-\!1)! k!}}
 \!\!\andf\!\!    \wt H_{2k-1}:=\frac{\sqrt2\,\cH_{2k-1}}{2^{k}\sqrt{(k\!-\!1)! k!}} \cdotp\end{equation}
Then, \eqref{eq:cL2k-1b} and \eqref{eq:cL2kb} become:
 $$ \begin{aligned}
 \frac12\frac d{dt}\wt L_{2k-1}^2 +\wt H_{2k-1}^2&\leq \frac12\biggl(\frac{k-\frac12}{k}\biggr)\wt H_{2k-2}^2+C_0 \biggl(\sum_{j=0}^{k-1} \wt L_{2j}^{1/2}\wt H_{2j}^{1/2}\wt L_{2k\!-\!2\!-\!2j}^{1/2}\wt H_{2k\!-\!2\!-\!2j}^{1/2}\biggr)\wt H_{2k},\\
  \frac12\frac d{dt}\wt L_{2k}^2 +\wt H_{2k}^2&\leq \
  \frac12\wt H_{2k-1}^2+C_0 \biggl(\sum_{j=1}^{k-1} \wt L_{2j}^{1/2}\wt H_{2j}^{1/2}\wt L_{2k\!-\!2j}^{1/2}\wt H_{2k\!-\!2j}^{1/2}\biggr)\wt H_{2k}. \end{aligned}$$
In order to get a nice control of the sum, we perform a second renormalization  as follows
for some suitable   nonnegative nondecreasing sequence $(c_m)_{m\in\N}$:
  \begin{equation}\label{eq:Lm}   \wt L_{2m-1}=c_m L_{2m-1},\quad  \wt H_{2m-1}=c_m  H_{2m-1},\quad \wt L_{2m}=c_m L_{2m}\andf   \wt H_{2m}=c_m H_{2m}.\end{equation}
The above inequalities translate into 
 $$ \begin{aligned}
 \frac12\frac d{dt} L_{2k-1}^2 +H_{2k-1}^2&\leq \frac12
 \frac{c_{k-1}^2}{c_k^2}H_{2k-2}^2
 +C_0 \biggl(\sum_{j=0}^{k-1}\frac{c_{j} c_{k-1-j}}{c_k} L_{2j}^{1/2}H_{2j}^{1/2} L_{2k\!-\!2\!-\!2j}^{1/2}H_{2k\!-\!2\!-\!2j}^{1/2}\biggr)H_{2k},\\
  \frac12\frac d{dt} L_{2k}^2 +H_{2k}^2&\leq \frac12 H_{2k-1}^2+C_0 \biggl(\sum_{j=1}^{k} 
  \frac{c_{j} c_{k-j}}{c_k}L_{2j}^{1/2} H_{2j}^{1/2} L_{2k\!-\!2j}^{1/2} H_{2k\!-\!2j}^{1/2}\biggr) H_{2k}.
    \end{aligned}$$
Let us take $c_j=(j!)^\alpha$ with  $\alpha>0$ so  that $  \frac{c_{j} c_{k-j}}{c_k}=\binom kj^{-\alpha}.$ 
Since $\frac{i}{k-j+i}\leq\frac jk$ for all $i\in\{1,\cdots,j\},$ we 
 have $\binom kj\geq \bigl(\frac kj\bigr)^j$   for all $j\in\{0,\cdots,k\}.$ Remembering that  $\binom kj=\binom k{k-j},$  we get
 $$  \frac{c_{j} c_{k-j}}{c_k} \leq \min\biggl(\Bigl(\frac jk\Bigr)^{\alpha j}, \Bigl(\frac{k-j}k\Bigr)^{\alpha(k-j)}\biggr)\cdotp$$
Using the obvious
 bound $j/k\leq1/2$ for $j\leq k/2,$ we conclude that 
\begin{equation}\label{eq:ccc0}\frac{c_{j} c_{k-j}}{c_k} \leq \min\bigl(2^{-j\alpha}, 2^{-(k-j)\alpha}\bigr)
\ \hbox{ for all }\ j\in\{0,\cdots,k\}.\end{equation}
Hence we have 
$$\sum_{j=0}^{k}   \frac{c_{j} c_{k-j}}{c_k}L_{2j}^{1/2} H_{2j}^{1/2} L_{2k\!-\!2j}^{1/2} H_{2k\!-\!2j}^{1/2}
  \leq \sum_{j=0}^k \sqrt{2^{-j\alpha} L_{2j} H_{2j}}   \sqrt{2^{-(k-j)\alpha} L_{2k-2j} H_{2k-2j}}.$$ 
  Similarly, as $c_{k-1}\leq c_k,$ we have
  $$\sum_{j=0}^{k-1}\frac{c_{j} c_{k-1-j}}{c_k} L_{2j}^{1/2}H_{2j}^{1/2} L_{2k\!-\!2\!-\!2j}^{1/2}H_{2k\!-\!2\!-\!2j}^{1/2}
  \leq \sum_{j=0}^{k-1}\sqrt{2^{-j\alpha} L_{2j}H_{2j}} \sqrt{2^{-(k-1-j)\alpha} L_{2k\!-\!2\!-\!2j}H_{2k\!-\!2\!-\!2j}}\cdotp$$
Hence, summing up the above two inequalities yields for all $k\geq1,$
\begin{multline}\label{eq:LHk}\frac12\frac d{dt}\Bigl( L_{2k-1}^2 + L_{2k}^2\Bigr) +\frac12H_{2k-1}^2+H_{2k}^2
\leq \frac12H_{2k-2}^2\\+C_0\biggl( \sum_{j=0}^k \sqrt{2^{-j\alpha} L_{2j} H_{2j}}   \sqrt{2^{-(k-j)\alpha} L_{2k-2j} H_{2k-2j}}\biggr)H_{2k}\\+C_0\biggl( \sum_{j=0}^{k-1} \sqrt{2^{-j\alpha} L_{2j} H_{2j}} 
  \sqrt{2^{-(k-1-j)\alpha} L_{2k-2j-2} H_{2k-2j-2}}\biggr)H_{2k}.\end{multline} 
Let us introduce   the notation:
$$ \mathbb L_m^2:=\sum_{k=0}^{m}  L_{k}^2\andf  \mathbb H_m^2:=\sum_{k=0}^{m}  H_{k}^2.$$
Then summing up \eqref{eq:H0} and \eqref{eq:LHk} from $k=1$ to $k=n$ gives after using the 
convolution inequality 
$$\sum_{k=0}^n\sum_{j=0}^n a_j b_{k-j} c_k \leq \|(a_j)\|_{\ell_n^{4/3}}  \|(b_j)\|_{\ell_n^{4/3}}  \|(c_j)\|_{\ell_n^{2}}
\with \ell^r_n:=\ell^r(\{0,\cdots,n\}),$$ 
\begin{equation}\label{eq:LL2n}
\frac12\frac d{dt}\L_{2n}^2+\frac12\H_{2n}^2 +\frac12H_{2n}^2\leq
2C_0\|(2^{-j\alpha} L_{2j} H_{2j})\|_{\ell_n^{2/3}}\H_{2n}.
\end{equation}
 H\"older inequality implies that
$$\begin{aligned}
\|(2^{-j\alpha} L_{2j} H_{2j})\|_{\ell_n^{2/3}}&\leq \|(2^{-j\alpha})\|_{\ell^2}  \|(L_{2j})\|_{\ell^2}  \|(H_{2j})\|_{\ell^2} \\
&\leq C_\alpha\L_{2n}\H_{2n}\with C_\alpha:=\sqrt{\frac1{1-2^{-2\alpha}}}\cdotp\end{aligned}$$
Hence, whenever $2C_0C_\alpha \L_{2n} \leq 1/4,$ we have
$$\frac d{dt}\L_{2n}^2+\frac12\H_{2n}^2 \leq0.$$
Since $\L_{2n}(0)=\|u_0\|_{L^2},$ a bootstrap argument allows to conclude that if 
\begin{equation}\label{eq:smalldata}8C_0C_\alpha\|u_0\|_{L^2}<1,\end{equation}
 then we have for all time $t\geq0$ and $n\in\N,$
$$\L_{2n}^2(t)+\frac12\int_0^t \H_{2n}^2(\tau)\,d\tau \leq \|u_0\|_{L^2}^2.$$
Applying the monotonous convergence theorem then leads  to \eqref{est:smalldata}.
\end{proof}


\section{The case of large data}

Here we want to establish time decay estimates for derivatives of $u$ at any order, in the case
of general, possibly large,  finite energy data. The main result is: 
\begin{thm}\label{thm:2} 
Let $\alpha>0.$ There exists a constant $C_\alpha$ depending only on $\alpha$ such that 
  for any initial  data $u_0$ in $L^2_\sigma$  and integer $n,$ we have  
  \begin{multline}\label{est:largedata}\sum_{k=0}^{n} \biggl(\frac{t^{2k}}{2^{2k}(k!)^{2+\alpha}} \|u_t^{(k)}(t)\|_{L^2}^2+
\frac{t^{2k+1}}{2^{2k+1} k! ((k+1)!)^{1+\alpha}} \|\nabla u_t^{(k)}(t)\|_{L^2}^2\biggr)\\+
\frac12\sum_{k=0}^n \int_0^t\biggl(\frac{\tau^{2k}}{2^{2k+1}(k!)^{2+\alpha}}\|\nabla u_\tau^{(k)}(\tau)\|_{L^2}^2
+\frac{\tau^{2k+1}}{2^{2k+1} k! ((k+1)!)^{1+\alpha}} \|u_\tau^{(k+1)}(t)\|_{L^2}^2\biggr)d\tau\\
\leq  C_\alpha^{2^n-1}
 \biggl(\|u_0\|_{L^2}^2  \exp\Bigl(\frac{C_0^2\|u_0\|_{L^2}^2}2\Bigr)\biggr)^{2^n},\end{multline}
  where $C_0$ stands for the optimal constant in \eqref{eq:lad}.
\end{thm}
\begin{proof} To handle the case of general   data,  we slightly modify \eqref{eq:cL2kb}.
In fact, starting from \eqref{eq:2k},  we use \eqref{eq:Rj2k} only for $j=0,\cdots,k-1$
and bound $\cR_{k,2k}$ as follows: 
$$\cR_{k,2k}\leq  \|\nabla u\|_{L^2}\|t^ku_t^{(k)}\|_{L^4}^2  \leq C_0\cH_0 \cL_{2k}\cH_{2k}.$$
This leads to 
$$\frac12 \frac d{dt} \cL_{2k}^2 +\cH_{2k}^2 \leq k \cH_{2k-1}^2 + C_0\sum_{j=0}^{k-1} \binom{k}{j}
    \cL_{2j}^{1/2}\cH_{2j}^{1/2}\cL_{2k-2j}^{1/2}\cH_{2k-2j}^{1/2}\cH_{2k}
   +C_0\cH_0\cL_{2k}\cH_{2k}.$$
Then, adding up  \eqref{eq:cL2k-1b} leads  
after the same succession of renormalizations as in the previous section to 
$$\displaylines{
\frac12\frac d{dt}(L_{2k}^2+L_{2k+1}^2)+H_{2k+1}^2+\frac12H_{2k}^2\leq\frac12 H_{2k-1}^2
+C_0\cH_0L_{2k}H_{2k}\hfill\cr\hfill
+C_0\biggl(\sum_{j=0}^{k}\sqrt{2^{-j\alpha} L_{2j}H_{2j}} \sqrt{2^{-(k-1-j)\alpha} L_{2k-2-2j}H_{2k-2-2j}}\biggr)
\hfill\cr\hfill
+C_0\biggl(\sum_{j=1}^{k-1}\sqrt{2^{-j\alpha} L_{2j}H_{2j}} \sqrt{2^{-(k-j)\alpha} L_{2k-2j}H_{2k-2j}}\biggr)\cdotp}
$$
 We may use for each $k$ that
$C_0\cH_0L_{2k}H_{2k}\leq \frac12H_{2k}^2 +\frac12 C_0^2\cH_0^2 L_{2k}^2$ and, 
after summing  for $k=1$ to $n,$ the second and third lines may be bounded
by $C_\alpha\L_{2n-2}\H_{2n-2}\H_{2n}.$ 
Hence, 
$$\frac d{dt}\L_{2n}^2+\H_{2n}^2\leq   C_0^2\cH_0^2 \L_{2n}^2+4 C_0C_\alpha \L_{2n-2}\H_{2n-2}\H_{2n},\qquad n\geq1.
$$
Performing the change of function:
$$\L_{2n}(t)= e^{\frac12\int_0^t C_0^2\cH_0^2(\tau)\,d\tau} \wt \L_{2n}(t)\andf
\H_{2n}(t)= e^{\frac12\int_0^t C_0^2\cH_0^2(\tau)\,d\tau} \wt \H_{2n}(t)$$
and using \eqref{eq:H0} yields
$$\frac d{dt}\wt\L_{2n}^2+\frac12\wt\H_{2n}^2\leq  \frac12A^2_{\alpha,u_0} \wt\L_{2n-2}^2\wt\H_{2n-2}^2$$
and thus, after time integration, 
\begin{equation}\label{eq:Ln}\wt\L_{2n}^2(t)+\frac12\int_0^t\wt\H_{2n}^2\,d\tau\leq\|u_0\|_{L^2}^2 +  \frac12A^2_{\alpha,u_0}\int_0^t \wt\L_{2n-2}^2\wt\H_{2n-2}^2\,d\tau.\end{equation}
Let us denote (assuming of course that $u_0\not=0$) 
$$X_n:=\|u_0\|_{L^2}^{-2}\biggl(\sup_{t\geq0} \wt\L_{2n}^2(t)+\frac12\int_0^\infty\wt\H_{2n}^2\,d\tau\biggr)$$
and use the inequality $2ab\leq(a+b)^2.$ Then, $X_0=1$ and \eqref{eq:Ln} can be rewritten as
$$X_n\leq 1+  \frac{A^2_{\alpha,u_0}\|u_0\|_{L^2}^2}2X_{n-1}^2,\qquad n\geq1.$$
Since obviously $X_n\geq1$ for all $n\in\N,$ we have
$$X_n\leq K X^2_{n-1}\with K:=1+ \frac{A^2_{\alpha,u_0}\|u_0\|_{L^2}^2}2,$$
which  implies that 
$$\forall n\in\N,\;  X_n\leq K^{2^n-1}$$ and thus
$$\wt\L_{2n}^2(t)+\frac12\int_0^t\wt\H_{2n}^2\,d\tau\leq\|u_0\|_{L^2}^2 \biggl(1+ \frac{A^2_{\alpha,u_0}\|u_0\|_{L^2}^2}2\biggr)^{2^n-1}.$$
Clearly, the computations here are  relevant  only if \eqref{eq:smalldata} is not satisfied, so that, up to an 
harmless change of $C_\alpha$ in the definition of $A_{\alpha,u_0},$ we have
$$1+\frac{A_{\alpha,u_0}^2\|u_0\|_{L^2}^2}{2}\leq  A_{\alpha,u_0}^2\|u_0\|_{L^2}^2.$$
In the end, we get a constant  $C_\alpha$  with behavior $\alpha^{-1/2}$ near $0$ such that for all $t\geq0,$
$$\L_{2n}^2(t)+\frac12\int_0^t\H_{2n}^2\,d\tau\leq C_\alpha^{2^n-1}
 \biggl(\|u_0\|_{L^2}^2  \exp\Bigl(\frac{C_0^2\|u_0\|_{L^2}^2}2\Bigr)\biggr)^{2^n}\cdotp$$
 This gives \eqref{est:largedata}. 
 \end{proof}
\begin{remark}\label{r:spacereg}  As a consequence of the regularity theory for the Stokes system, 
in the case where the domain $\Omega$ is smooth with a `reasonable shape' (like e.g. bounded simply 
connected or exterior domains), then one can deduce  decay estimates at any order 
for the \emph{space} derivatives of~$u.$ \end{remark}
 Indeed, we have $u|_{\d\Omega}=0,$ 
$$-\Delta u+\nabla P=-u_t-u\cdot\nabla u \andf \div u=0\quad\hbox{in}\quad\Omega.$$
Hence there exists a constant $C$ depending only on $\Omega$ (if it is e.g uniformly $C^2$ and bounded) such that:
$$
\|\nabla^2 u\|_{L^2}+\|\nabla P\|_{L^2}\leq C\|u_t+u\cdot\nabla u\|_{L^2}.$$
Multiplying by $t$ and using H\"older and Ladyzhenskaya inequality yields
$$
\|t\nabla^2 u\|_{L^2}+\|t\nabla P\|_{L^2}\leq C\|tu_t\|_{L^2} +CC_0\|u\|_{L^2}^{1/2}\|\sqrt t\nabla u\|_{L^2}\|t\nabla^2u\|_{L^2}^{1/2}.
$$
Using Young inequality allows to conclude that for some constant still denoted by $C,$
$$\|t\nabla^2 u\|_{L^2}+\|t\nabla P\|_{L^2}\leq C\bigl(\|tu_t\|_{L^2}+ \|u\|_{L^2}\|\sqrt t\,\nabla u\|_{L^2}^2\bigr)\cdotp$$
This allows to get a uniform bound of the left-hand side in terms of $\|u_0\|_{L^2},$ 
due to \eqref{est:largedata} with $n=1.$
\medbreak
By the same token,  it is easy to bound $\|t^{k+1}\nabla^2 u_t^{(k)}\|_{L^2}$ for any integer $k$ since
 $(u_t^{(k)},P_t^{(k)})$ satisfies $u^{(k)}_t|_{\d\Omega}=0,$ 
$$-\Delta u_t^{(k)}+\nabla P_t^{(k)}=-u_t^{(k+1)}-\sum_{j=0}^k\binom kj u_t^{j}\cdot\nabla u_t^{(k-j)} \andf \div u_t^{(k)}=0\quad\hbox{in}\quad\Omega.$$
Hence,  requiring only $C^2$ regularity for $\Omega$ gives 
$$\displaylines{\|t^{k+1}\nabla^2 u_t^{(k)}\|_{L^2}+\|t^{k+1}\nabla P_t^{(k)}\|_{L^2}\leq 
C\biggl(\|t^{k+1}u_t^{(k+1)}\|_{L^2}\hfill\cr\hfill+\sum_{j=0}^k\binom kj \|(t^{j+\frac14}u_t^{j})\cdot
(t^{k-j+\frac34}\nabla u_t^{(k-j)})\|_{L^2}\biggr)\cdotp}$$
The right-hand side may be bounded in terms of $\|u_0\|_{L^2}$ 
 by combining H\"older inequality, \eqref{eq:lad} and   \eqref{est:largedata} with $n=k+1.$  
\smallbreak
In order to bound higher order space derivatives, we use that if $\Omega$ is smooth then, for all $j\in\N,$  there exists a constant $C_j$ depending only on $\Omega$ and $j$ such that
$$\|\nabla^{j+2}u\|_{L^2} + \|\nabla^{j+1}P\|_{L^2} \leq C_j\bigl(\|\nabla^ju_t\|_{L^2}+\|\nabla^j(u\cdot\nabla u)\|_{L^2}\bigr)\cdotp$$
Similar inequalities at any order may be written  for $\nabla^{j+2}u_{t}^k.$
Then using a careful induction argument allows to bound 
$t^{k+j/2}\|\nabla^j u_t^{(k)}\|_{L^2}$ in terms of $u_0$ at any order.
The (tedious) verifications are left to the reader.


\section{Small time Gevrey regularity  in the case  of large data}

In this section we address the question of Gevrey regularity in the case where $u_0$ is large. 
Since the solution $(\ell,Q)$ to the Stokes system \eqref{eq:stokes} 
has analytic regularity (recall \eqref{est:stokes}), 
it suffices to study the regularity of  the fluctuation $f:= u-\ell$ that, by definition,  
 satisfies $f|_{\d\Omega}=0,$  $f|_{t=0}=0,$ and, for some scalar function $R,$ 
\begin{equation}\label{eq:fluctuation} \left\{\begin{aligned} &f_t-\Delta f+\nabla R=-u\cdot\nabla u &\hbox{in }\  \R_+\times\Omega,\\
&\div f =0&\hbox{in }\  \R_+\times\Omega.\end{aligned}\right.\end{equation}
The main result of this section reads:
\begin{thm}\label{thm:3} Let $\alpha>0.$ There exists a positive constant $C_\alpha,$ a positive time $T_{\alpha,u_0}$
 and a continuous increasing function $\phi_{\alpha,u_0}:[0,T_{\alpha,u_0}]\to\R_+$ vanishing at $0$  such that
the fluctuation $f$ satisfies for all $t\in[0,T_{\alpha,u_0}]$:
$$\displaylines{\sum_{k=0}^{\infty} \biggl(\frac{t^{2k}}{2^{2k}(k!)^{2+\alpha}} \|f_t^{(k)}(t)\|_{L^2}^2+
\frac{t^{2k+1}}{2^{2k+1} k! ((k+1)!)^{1+\alpha}} \|\nabla f_t^{(k)}(t)\|_{L^2}^2\biggr)\hfill\cr\hfill+
\sum_{k=0}^\infty \int_0^t\biggl(\frac{\tau^{2k}}{2^{2k+1}(k!)^{2+\alpha}}\|\nabla f_\tau^{(k)}(\tau)\|_{L^2}^2
+\frac{\tau^{2k+1}}{2^{2k+1} k! ((k+1)!)^{1+\alpha}} \|f_\tau^{(k+1)}(\tau)\|_{L^2}^2\biggr)d\tau\leq \phi_{\alpha,u_0}(t).}$$
\end{thm}
\begin{proof} 
Denote by $L_m^\ell$ and $H_m^\ell$ (resp. $L_m^f$ and $H_m^f$) the quantities $L_m$ and $H_m$
defined in \eqref{eq:Lm}  pertaining to $\ell$ (resp. $f$). According to Leibniz rule, we have  for all $k\in\N,$  
$$f_t^{(k)} -\Delta f_t^{(k-1)} + \nabla R_t^{k-1}= - \sum_{j=0}^{k-1} \binom{k-1}{j} u_t^{(j)}\cdot\nabla u_t^{(k-1-j)}.$$
Hence taking the scalar product with $t^{2k-1}f_t^{(k)}$ (odd case) 
or with $t^{2k} f_t^{(k)}$ (even case) and using the  following type of inequalities: 
$$\displaylines{\|t^j u_t^{(j)}\|_{L^4} \|t^{k-j} u_t^{(k-j)}\|_{L^4} \leq
\|t^j \ell_t^{(j)}\|_{L^4} \|t^{k-j} \ell_t^{(k-j)}\|_{L^4} +
\|t^j \ell_t^{(j)}\|_{L^4} \|t^{k-j} f_t^{(k-j)}\|_{L^4} \hfill\cr\hfill
+\|t^j f_t^{(j)}\|_{L^4} \|t^{k-j} \ell_t^{(k-j)}\|_{L^4} 
+\|t^j f_t^{(j)}\|_{L^4} \|t^{k-j} f_t^{(k-j)}\|_{L^4}} $$
which implies, thanks to \eqref{eq:lad} that 
$$\displaylines{\|t^j u_t^{(j)}\|_{L^4} \|t^{k-j} u_t^{(k-j)}\|_{L^4} \leq C_0\Bigl(\sqrt{L^\ell_{2j}H^\ell_{2j}L^\ell_{2k-2j}H^\ell_{2k-2j}}
+\sqrt{L^\ell_{2j}H^\ell_{2j}L^f_{2k-2j}H^f_{2k-2j}}\hfill\cr\hfill
+\sqrt{L^f_{2j}H^f_{2j}L^\ell_{2k-2j}H^\ell_{2k-2j}}
+\sqrt{L^f_{2j}H^f_{2j}L^f_{2k-2j}H^f_{2k-2j}}\Bigr),}$$
 the counterpart of \eqref{eq:LL2n} now reads
 $$\displaylines{\frac12\frac d{dt}(\L_{2n}^f)^2+\frac12(\H_{2n}^f)^2 +\frac12(H_{2n}^f)^2\leq
2C_0\|(2^{-j\alpha} L_{2j}^f H_{2j}^f)\|_{\ell_n^{2/3}}\H_{2n}^f\hfill\cr\hfill
+4C_0\sqrt{\|(2^{-j\alpha} L_{2j}^\ell H_{2j}^\ell)\|_{\ell_n^{2/3}}\|(2^{-j\alpha} L_{2j}^f H_{2j}^f)\|_{\ell_n^{2/3}}}\;\H_{2n}^f+
2C_0\|(2^{-j\alpha} L_{2j}^\ell H_{2j}^\ell)\|_{\ell_n^{2/3}}\H_{2n}^f}$$
with $$ \mathbb L_m^p:=\sqrt{\sum_{k=0}^{m}  (L_{k}^p)^2}\andf  \mathbb H_m^p:=\sqrt{\sum_{k=0}^{m} (H_{k}^p)^2}
\quad\hbox{for }\ p\in\{f,\ell\}.$$
Using the Young inequality to bound the right-hand side, this inequality implies that
 $$\begin{aligned}\frac d{dt}(\L_{2n}^f)^2+(\H_{2n}^f)^2 +(H_{2n}^f)^2&\leq
8C_0\|(2^{-j\alpha} L_{2j}^f H_{2j}^f)\|_{\ell_n^{2/3}}\H_{2n}^f
+8C_0\|(2^{-j\alpha} L_{2j}^\ell H_{2j}^\ell)\|_{\ell_n^{2/3}}\H_{2n}^f\\
&\leq 8C_0C_\alpha \L_{2n}^f(\H_{2n}^f)^2 +  8C_0C_\alpha \L_{2n}^\ell\H_{2n}^\ell\H_{2n}^f\\
&\leq\Bigl(\frac14+8C_0C_\alpha \L_{2n}^f\Bigr)(\H_{2n}^f)^2+64 C_0^2C_\alpha^2(\L_{2n}^\ell \H_{2n}^\ell)^2.
\end{aligned}$$
Therefore, whenever 
\begin{equation}\label{eq:condition}8C_0C_\alpha\L_{2n}^f(t)\leq1/4,\end{equation} we have
$$(\L_{2n}^f(t))^2+\frac12\int_0^t(\H_{2n}^f)^2d\tau \leq 64 C_0^2C_\alpha^2
\int_0^t\bigl(\L^\ell_{2n}\H^\ell_{2n}\bigr)^2d\tau
\leq 64 C_0^2C_\alpha^2\|u_0\|_{L^2}^2\int_0^t\bigl(\H^\ell_{2n}\bigr)^2d\tau.$$
Since \eqref{est:stokes} guarantees that 
$$\int_0^\infty \sum_{k=0}^\infty(H_k^\ell)^2\,dt<\infty,$$
Lebesgue dominated convergence theorem ensures that there exists $T_0>0$ such that 
$$
8C_0C_\alpha\|u_0\|_{L^2}\sqrt{\int_0^{T_0} \sum_{k=0}^\infty(H_k^\ell)^2\,dt}<\frac1{32C_0C_\alpha}\cdotp$$
Reverting to the above inequality and bootstrapping, one can now conclude that \eqref{eq:condition} is satisfied on $[0,T_0]$ for 
all $n\in\N,$ and that we thus have for all $t\in[0,T_0],$ 
$$\sum_{k=0}^{\infty}  (L_{k}^f(t))^2+\frac12\int_0^t \sum_{k=0}^{\infty} \bigl(H_{k}^f\bigr)^2d\tau
\leq 64 C_0^2C_\alpha^2\|u_0\|_{L^2}^2\int_0^t\sum_{k=0}^{\infty}  (H_{2k}^\ell)^2d\tau.$$
As the right-hand side is a continuous nondecreasing function vanishing at zero, this completes the proof.
\end{proof}


\section{Faster decay} 

In this last section,  we assume that there exist $K \geq 0$ and $\gamma> 0$ such that our
reference solution satisfies
\begin{equation}\label{eq:faster}
\|u(t)\|_{L^2} \leq Kt^{-\gamma},\qquad t>0.\end{equation}
It is known that \eqref{eq:faster} holds true with  $\gamma=1/2$ if $u_0$ is in $L^1$
(see \cite{KM}). 
  Fix some  $\alpha>0$ and set for all 
$k\in\N$ and $t\geq t_0\geq0,$
$$
\begin{aligned}
L_{2k}(t,t_0):=\frac{\|(t-t_0)^ku_t^{(k)}(t)\|_{L^2}}{2^k(k!)^{1+\alpha}}
&\andf&H_{2k}(t,t_0):=\frac{\|(t-t_0)^k\nabla u_t^{(k)}(t)\|_{L^2}}{2^k(k!)^{1+\alpha}},\\
L_{2k+1}(t,t_0):=\frac{\|(t-t_0)^{k+\frac12}\nabla u_t^{(k)}(t)\|_{L^2}}{2^k\sqrt{k!(k\!+\!1)}\,((k\!+\!1)!)^{\alpha}}
&\andf&H_{2k+1}(t,t_0):=\frac{\|(t-t_0)^{k+\frac12}u_t^{(k+1)}(t)\|_{L^2}}{2^k\sqrt{k!(k\!+\!1)!}\,((k\!+\!1)!)^{\alpha}}\cdotp
\end{aligned}
$$
Then,  repeating the computations leading to \eqref{est:smalldata}, we arrive at
$$ \L_{2n}^2(t,t_0) + \frac12\int_0^t\H_{2n}^2(\tau,t_0)\,d\tau\leq\|u(t_0)\|_{L^2}^2$$
with 
$$\L_{2n}^2(t,t_0):=\sum_{k=0}^{2n}L^2_k(t,t_0)\andf\H_{2n}^2:=\sum_{k=0}^{2n}H_k^2(t,t_0)
$$
whenever $8C_0C_\alpha \|u(t_0)\|_{L^2} \leq 1.$ 
\medbreak
Clearly, this latter condition is satisfied for any $t_0 \geq 0$ if $8C_0C_\alpha \|u_0\|_{L^2}\leq 1,$ or, due to \eqref{eq:faster}, 
 at $t_0 = t/2$
if $t \geq 2(8C_0C_\alpha K)^{1/\gamma}$  in the general case. Consequently, we have proved the following statement:
\begin{thm}\label{thm:4}
Let $\alpha>0.$ Assume that the considered finite energy global solution $u$ to (NS) satisfies \eqref{eq:faster}. 
Then there exists $t_0\geq0$ such that for all $t\geq t_0$ we have,
\begin{multline}
\sum_{k=0}^\infty\biggl(\frac{t^{2k+2\gamma}}{2^{4k}(k!)^{2+\alpha}}\|u_t^{(k)}(t)\|_{L^2}^2 + \frac{t^{2k+1+2\gamma}}{2^{4k+1}k!
((k\!+\!1)!)^{1+\alpha}}
\|\nabla u_t^{(k)}(t)\|_{L^2}^2\biggr)\\+
\sum_{k=0}^\infty\int_0^t\biggl(\frac{\tau^{2k+2\gamma}}{2^{4k}(k!)^{2+\alpha}}\|\nabla u_\tau^{(k)}(\tau)\|_{L^2}^2 + \frac{\tau^{2k+1+2\gamma}}{2^{4k+1}k!((k\!+\!1)!)^{1+\alpha}}
\|u_\tau^{(k+1)}(\tau)\|_{L^2}^2\biggr)d\tau\leq 2^{2\gamma}K^2.\end{multline}
\end{thm}


 \begin{small}

\end{small}



\end{document}